\documentclass[12pt]{amsart}
\usepackage[a4paper]{geometry}
\usepackage{amsmath, amsfonts, amssymb, color, url, tikz}
\usepackage[russian, english]{babel}
\newtheorem{theorem}{Theorem}

\newtheorem{lemma}[theorem]{Lemma}

\begin{document}

\title[KdV on a tree with unbounded root and edges]{Linearized Korteweg -- De Vries equation on a tree with unbounded root and edges}
\author{Maqsad I. Akhmedov$^*$ }
\author{Doniyor Babajanov$^\maltese$}
\author{Marks Ruziboev$\dagger$}
\address{$*$Yeoju Technical Institute in Tashkent, 156 Usman Nasyr Str., 100121, Tashkent}
\email{maqsad.ahmedov@mail.ru}
\address{$^\maltese$Ajou University in Tashkent, 113 Asalobod Str., 100204, Tashkent; \newline AKFA University, 10 Kukcha Darvoza 1-deadlock Street, Shaykhantokhur district, Tashkent}
\email{d.b.babajanov@gmail.com}
\address{$\dagger$Faculty of Mathematics, University of Vienna, Oskar Morgensternplatz 1 1090 Vienna, Austria}
\email{ruziboev@univie.ac.at}

\begin{abstract}
We investigate the linearized KdV equation on a metric tree consisting of three different types of bonds:  incoming unbounded root, two finite bonds, and four outgoing unbounded bonds. Under natural assumptions at the vertices, we obtain the uniqueness of a solution. To show the existence we use the theory of potentials and reduce the problem to a system of linear algebraic equations. We show that the latter is uniquely solvable under conditions of the uniqueness theorem. Also, we show that the system we consider can be used to model wave propagation in pipelines.
% in the class of Schwartz functions {\color{red} and in Sobolev classes}.
\end{abstract}

\thanks{The research of MR is funded by Lise Meitner Fellowship FWF M-2816 of the Austrian Research Fund FWF}

\maketitle

\section{Introduction}
The Korteweg-de Vries (KdV) equation
%\[u_t+u_{xxx}+uu_x=0\]
was first obtained in \cite{KdV} as a model for the  propagation of  long waves on  shallow water surface. Since then, it has attracted much attention of both physicists and mathematicians. It is involved in a models of a wide variety of physical processes especially those exhibiting shock waves, traveling waves, and solitons. It is used in fluid dynamics, aerodynamics, and continuum mechanics as a model for shock wave formation, solitons, turbulence, boundary layer behavior, and mass transport. It has been studied and applied for many decades.  The KdV equation has many remarkable properties, including the property discovered by Gardner \cite{Gard}: it can be solved exactly, as an initial-value problem, starting with arbitrary initial data in a suitable space. This discovery was revolutionary, and it drew the interest of many scholars. We note especially the work of Zakharov and Faddeev \cite{ZF}, who showed that the KdV equation is a nontrivial example of an infinite-dimensional Hamiltonian system that is completely integrable. The KdV equation is particularly notable as the prototypical example of an exactly solvable model.  For instance, one way to solve the KdV equation is to use the inverse scattering transform. The solutions in turn include prototypical examples of solitons. On the other hand, one can study the linearized KdV, which provides an asymptotic description of linear, unidirectional, weakly dispersive long waves, for example, shallow water waves. Earlier, it was proven that via normal form transforms, the solution of the KdV equation can be reduced to the solution for the linear KdV equation \cite{Whit}.  Belashov and Vladimirov \cite{BV} numerically investigated the evolution of a single disturbance $u\left( 0,x \right)={{u}_{0}}\exp (-{{x}^{2}}/{{l}^{2}})$ and showed that in the limit $l\to 0,{{u}_{0}}{{l}^{2}}=const,$ the solution of the KdV equation is qualitatively similar to the solution of the linearized KdV equation. Boundary value problems on half lines were considered in \cite{BF, CC, Dj, FS, Hol}\footnote{We don't aim in this notes to give  a complete  account of progress in this rich theory and so our list of references by no means is exhaustive.}.

In recent years partial differential equations on metric graphs, known as networks, became an object of extensive study as they arise in certain branches of physics and biology. In this direction the study  was mainly concentrated on Schr\"odinger equation. In recent works \cite{SUA, SAU, ASE} solutions of linearized  KdV on star graphs were obtained, while  local well-posedness  for KdV on star graphs were studied \cite{AC, CC, C}.

In this paper, we investigate the linearized KdV equation on metric trees  which contains three different bonds. The bonds are denoted by ${B}_{j},\,\,j=\overline{1,7}$ the coordinate ${x}_{1}$on ${B}_{1}$ is defined from $-\infty$ to $0$, and coordinates ${x}_{2}$ and ${x}_{3}$ from links ${B}_{2}$ and ${B}_{3}$ from $0$ to $L$, and coordinates ${x}_{4}$ and ${x}_{5}$ and ${x}_{6}$ and ${x}_{7}$ on the bonds ${B}_{4}$ and ${B}_{5}$ and ${B}_{6}$ and ${B}_{7}$ are defined from $0$ to such that on each bond the vertex corresponds to $0$. We denote the graph by $\Gamma$. On each bond we consider the linear equation:
\begin{equation}\label{(1)}
\left( \frac{\partial }{\partial t}-\frac{{{\partial }^{3}}}{\partial x_{j}^{3}} \right){{u}_{j}}({{x}_{j}},t)={{f}_{j}}(x,t),\ \ \ t>0,{{x}_{j}}\in {{B}_{j}},j=\overline{1,7}.
\end{equation}
Below, we will also use the notation $x$ instead of ${{x}_{j}},\,\,j=\overline{1,7}$. We treat a boundary value problem and using the method of potentials, reduce it to a system of integral equations.The solvability of the obtained system of integral equations is proven.

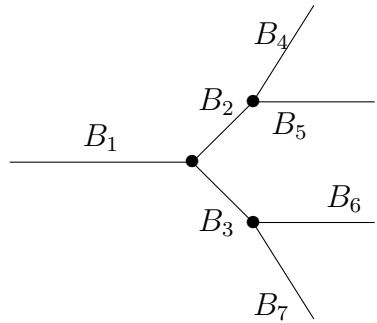
\begin{figure}[h]
\begin{tikzpicture}[scale =0.8]
\draw(-3, 0)--(0,0)(0,0)--(1,1)(1,1)--(2,2.6)(1,1)--(3,1)
(0,0)--(1,-1)(1,-1)--(2, -2.6)(1,-1)--(3, -1);
\node[above] at (-1.5,0){$B_1$};
\node[above] at (.4,0.6){$B_2$};
\node[above] at (1.3,1.7){$B_4$};
\node[below] at (1.6,1){$B_5$};
\node[below] at (.4,-0.6){$B_3$};
\node[above] at (2.5, -1){$B_6$};
\node[below] at (1.3, -2){$B_7$};
\node at (0,0){$\bullet$};
\node at (1,1){$\bullet$};
\node at (1,-1){$\bullet$};
\end{tikzpicture}
\caption{The graph $\Gamma$. The bonds $B_2$ and $B_3$ are compact intervals, $B_1$ is a copy of $(-\infty, 0]$ and $B_4-B_7$ are copies of $[0, +\infty)$.}
\label{graph}
\end{figure}

\section{Formulation of the problems}
We look for solutions  $u\in C^{3,1}(\Gamma\times \mathbb R)$\footnote{This is class of functions that are $C^2$ in $x$ and $C^1$ in $t$.}  vanishing  at $-\infty$, i.e.
\begin{equation}\label{eq:alpha}
\lim_{x\to -\infty} u(x, t)=\lim_{x\to -\infty}\partial_{x}u(x, t)=0, \text{ for all }\,\, t>0,
\end{equation}
and set the Kirchhoff conditions to connect the functions $u_j(x_j)$ defined on each bond of the graph.
\begin{align}
u_1(0,t)=a_2u_2&(0,t)=a_3u_3(0,t), u_{1x}(0,t)=b_2u_{2x}(0,t)=b_3u_{3x}(0,t),\label{(2)}\\
&u_{1xx}(0,t)=\frac{1}{a_2}u_{2xx}(0,t)+\frac{1}{a_{3}}u_{3xx}(0,t),\label{(3)}\\
a_2u_2(L,t)=a_{4}u_4&(0,t)=a_{5}u_{5}(0,t),
b_2u_{2x}(L,t)=b_{4}u_{4x}(0,t)=b_{5}u_{5x}(0,t),\label{(4)}\\
&\frac{1}{a_2}u_{2xx}(L,t)=\frac{1}{a_4}u_{4xx}(0,t)+\frac{1}{a_{5}}u_{5xx}(0,t),\label{(5)}\\
a_3u_3(L,t)=a_6u_6&(0,t)=a_7u_7(0,t), b_3u_{3x}(L,t)=b_6u_{6x}(0,t)=b_7u_{7x}(0,t),\label{(6)}\\
&\frac{1}{a_3}u_{3xx}(L,t)=\frac{1}{a_6}u_{6xx}(0,t)+\frac{1}{a_{7}}u_{7xx}(0,t),\label{(7)}
\end{align}
For $t>0$, where $a_k, b_k, k=\overline{2,7},$ are nonzero constants. Moreover, we assume that the ${{f}_{j}}(x,t), j=\overline{1,7}$ and the initial conditions:
\begin{equation}
u_j(x,0)=u_0(x),k, x\in \overline{B_j}, j=\overline{1,7}, \label{(8)}
\end{equation}
are smooth enough and bounded, and that   satisfies the vertex conditions
(2)-(7). These vertex conditions are not the only possible ones, and the main motivation for our choice is that they guarantee conservation of energy of the solution and, if the solutions decay (to zero) at infinity, the norm (energy) conservation.

Existence and uniqueness of solutions:
\begin{theorem}\label{main} Suppose that
 $$a_{2i}b_{2i}a_{2i+1}b_{2i+1}+\frac{a_{2i+1}b_{2i+1}(a_{2i}+b_{2i})}{a_{2i}}+\frac{a_{2i}b_{2i}(a_{2i+1}+b_{2i+1})}{a_{2i+1}}\neq 0.$$
 for $ i=1,2,3$. Then the problem (1)-(8) has a unique  solution in $C^{3,1}(\Gamma\times \mathbb R)$.
\end{theorem}

We first prove the uniqueness, which is the result of the following 
\begin{lemma} Suppose that $\frac{1}{b_{2i}^{2}}+\frac{1}{b_{2i+1}^{2}}\leq 1, i=1,2,3.$ Then the problem (1)--(8) has at most one solution.
\end{lemma}
\begin{proof} The uniqueness of the solutions follows the classical path. Namely, we assume by contraction  that there are two solutions $v_k$ and $w_k$ to $(1)-(8)$ on $B_j$, $j=1,..., 7$.  By linearity, $u_k=v_k-w_k$ is a solution of (1) with $f_k=0$.  We will show that $u\equiv 0$ as follows.
Multiplying  both sides of $(1)$ by $u_{k},$ and integrating on $B_{k},$
we have
\[
\frac12\int_{B_k}\frac{\partial u_k^2} {\partial t}dx=\int_{B_k}u_k\frac{\partial u_k} {\partial x^3}dx.
\]
Integrating by parts the right hand side of the latter equation we obtain
\[
\frac12\frac{\partial}{\partial t}\int_{B_k}u_k^2dx=u\cdot\frac{\partial u_k} {\partial x^2}\mid_{\partial B_k}- \int_{B_k}\frac{\partial} {\partial x}u_{kx}^2dx.
\]
For $k=1$ using \eqref{eq:alpha}  we obtain
\[
\frac12\frac{\partial}{\partial t}\int_{B_1}u_1^2dx=u_1(0, t)u_{1xx}(0, t)-\frac12u^2_{1x}(0, t).
\]
Similarly, for $k=2,3$ we have
\[
\frac12\frac{\partial}{\partial t}\int_{B_k}u_k^2dx=u_k(L, t)u_{kxx}(L, t)-u_1(0, t)u_{1xx}(0, t)-\frac12u^2_{kx}(L, t)+\frac12u^2_{kx}(0, t).
\]
Finally, for $k=4,5,6,7$ we have
\begin{align*}
\frac12\frac{\partial}{\partial t}\int_{B_k}u_k^2dx=&\lim_{L\to\infty}\left(u_k(L, t)u_{kxx}(L, t)-\frac12u^2_{kx}(L, t)\right)\\
&+\frac12u^2_{kx}(0, t)-u_1(0, t)u_{1xx}(0, t).
\end{align*}
Defining $\|u\|^{2}_{\Gamma}=\sum^{7}_{k=1}\int_{B_{k}}u^{2}_{k}(x)dx$  and using the conditions \eqref{(2)}-\eqref{(8)} and summing over $k$ the above three equations, we obtain
$$\frac{\partial}{\partial t}\|u\|^{2}_{\Gamma}=2(\frac{1}{b_{2i}^{2}}+\frac{1}{b_{2i+1}^{2}}-1)u^{2}_{1x}(0,t)+2\sum^{7}_{k=1}\int_{B_{k}}f_{k}(x,t)u_{k}(x,t)dx\leq 2\|f\|_{\Gamma}\|u\|_{\Gamma}.$$ Since $f_k=0$ in our case, it follows that $\frac{d}{dt}\|u\|_{\Gamma}\leq 0$ and hence, uniqueness follows.

\end{proof}

%$$a_{4}b_{4}a_{5}b_{5}+\frac{a_{5}b_{5}(a_{4}b_{2}+a_{2}b_{4})}{a_{4}}+\frac{a_{4}b_{4}(a_{5}b_{2}+a_{2}b_{5})}{a_{5}}\neq 0,$$ $$a_{6}b_{6}a_{7}b_{7}+\frac{a_{7}b_{7}(a_{6}b_{3}+a_{3}b_{6})}{a_{6}}+\frac{a_{6}b_{6}(a_{7}b_{3}+a_{3}b_{7})}{a_{7}}\neq 0,$$
%$$ \frac{1}{b_{2i}^{2}}+\frac{1}{b_{2i+1}^{2}}\leq 1, i=1,2,3$$

\begin{proof}[Proof of Theorem \ref{main} (Existence)]%Proof of theorem.
It remains to prove the existence of solutions. To prove it we we use the following  fundamental solutions of the equation $u_{t}-u_{xxx}=0$
$$
U(x,t;\xi, \eta)=\begin{cases}
\frac{1}{(t-\eta)^{\frac13}}f\left(\frac{x-\xi}{(t-\eta)^\frac{1}{3}}\right), & t>\eta,\\ 0, &t\leq \eta, \end{cases}
$$
$$
V(x,t;\xi, \eta)=\begin{cases}
\frac{1}{(t-\eta)^\frac{1}{3}}\varphi\left(\frac{x-\xi}{(t-\eta)^\frac13}\right), &t>\eta,\\
0, & t\leq \eta,\end{cases}
$$
where $f(x)=\frac{\pi }{\sqrt[3]{3}}Ai\left( -\frac{x}{\sqrt[3]{3}} \right),\varphi (x)=\frac{\pi }{\sqrt[3]{3}}Bi\left( -\frac{x}{\sqrt[3]{3}} \right)$ for $x\ge 0,$ $\varphi (x)=0$ for $x<0$.  Recall that $Ai(x)$ and $Bi(x)$ are the Airy functions\footnote{Airy's functions are defined by improper integrals
$$Ai(x)=\frac{1}{\pi}\int_{0}^{\infty}cos(\frac{t^{3}}{3}+xt)dt, Bi(x)=\frac{1}{\pi}\int_{0}^{\infty}exp(-\frac{1}{3}t^{3}+xt)dt+\frac{1}{\pi}\int_{0}^{\infty}sin(\frac{t^{3}}{3}+xt)dt.$$
We use the following
$$Ai(0)=\frac{1}{3^{\frac{2}{3}}\Gamma(\frac{2}{3})}, Ai'(0)=-\frac{1}{3^{\frac{1}{3}}\Gamma(\frac{1}{3})},\quad
Bi(0)=\frac{1}{3^{\frac{1}{6}}\Gamma(\frac{2}{3})}, Bi'(0)=\frac{1}{3^{-\frac{1}{6}}\Gamma(\frac{1}{3})}.$$}.
The functions $f(x)$ and $\varphi (x)$ are integrable and
$$\int\limits_{-\infty }^{0}{f(x)dx}=\frac{\pi }{3},\int\limits_{0}^{+\infty }{f(x)dx}=\frac{2\pi }{3},\int\limits_{0}^{+\infty }{\varphi (x)dx}=0.$$

Below, we also use fractional integrals [8]:
$$J_{0,t}^{\alpha}f(t):=\frac{1}{\Gamma(\alpha)}\int_{0}^{t}(t-\tau)^{\alpha-1}f(\tau)d\tau, 0<\alpha <1$$
and the inverse of this operator, i.e., the Riemann-Lioville fractional derivatives [8,9], defined by:
$$D_{0,t}^{\alpha}f(t):=\frac{1}{\Gamma(1- \alpha)}\frac{d}{dt}\int_{0}^{t}(t-\tau)^{-\alpha}f(\tau)d\tau, 0<\alpha <1.$$
We look for solution in the form:
\begin{align*}
u_{1}(x,t)=&\int_{0}^{t}U(x,t;0,\eta)\varphi_{1}(\eta)d\eta+F_{1}(x,t);\\
u_{k}(x,t)=&\int_{0}^{t}U(x,t;0,\eta)\varphi_{k}(\eta)d\eta+\int_{0}^{t}U(x,t;L,\eta)\alpha_{k}(\eta)d\eta\\
&\qquad+\int_{0}^{t}V(x,t;0,\eta)\beta_{k}(\eta)d\eta+F_{k}(x,t), k=2,3;\\
u_{i}(x,t)=&\int_{0}^{t}U(x,t;0,\eta)\varphi_{i}(\eta)d\eta+\int_{0}^{t}V(x,t;0,\eta)\psi_{i}(\eta)d\eta+F_{i}(x,t), i=4,5,6,7,
\end{align*}
where
\begin{equation}
F_{k}(x,t)=\frac{1}{\pi}\int_{0}^{t}\int_{B_{k}}U(x,t;\xi,\eta)f_{k}(\xi,\eta)d\xi d\eta, k=\overline{1,7}. \label{(10)}
\end{equation}
Satisfying the conditions \eqref{(2)}, we have:
$$\int_{0}^{t}\frac{1}{(t-\eta)^{\frac{1}{3}}}f(0)\varphi_{1}(\eta)d\eta+F_{1}(0,t)=\int_{0}^{t}\frac{a_{2}}{(t-\eta)^{\frac{1}{3}}}f(0)\varphi_{2}(\eta)d\eta+$$
$$\int_{0}^{t}\frac{a_{2}}{(t-\eta)^{\frac{1}{3}}}\varphi(0)\beta_{2}(\eta)d\eta+\int_{0}^{t}\frac{a_{2}}{(t-\eta)^{\frac{1}{3}}}f(0)f(\frac{-L}{(t-\eta)^{\frac{1}{3}}})\alpha_{2}(\eta)d\eta+a_{2}F_{2}(0,t)=$$
$$\int_{0}^{t}\frac{1}{(t-\eta)^{\frac{1}{3}}}f(0)\varphi_{1}(\eta)d\eta+F_{1}(0,t)=\int_{0}^{t}\frac{a_{3}}{(t-\eta)^{\frac{1}{3}}}f(0)\varphi_{3}(\eta)d\eta+$$
$$\int_{0}^{t}\frac{a_{3}}{(t-\eta)^{\frac{1}{3}}}\varphi(0)\beta_{3}(\eta)d\eta+\int_{0}^{t}\frac{a_{3}}{(t-\eta)^{\frac{1}{3}}}f(0)f(\frac{-L}{(t-\eta)^{\frac{1}{3}}})\alpha_{3}(\eta)d\eta+a_{3}F_{3}(0,t).$$
 Taking  derivatives of order $1/3$ in the above equality and using the properties of fractional derivatives we obtain
\begin{align}
f(0)\varphi_{1}(t)-a_{2}f(0)\varphi_{2}(t)-a_{2}\varphi(0)\beta_{2}(\eta)+\int_{0}^{t}K_{1}\alpha_{2}(\eta)d\eta=\nonumber\\
=\frac{1}{\Gamma(\frac{1}{3})}D_{(0,t)}^{\frac{2}{3}}[a_{2}F_{2}(0,t)-F_{1}(0,t)], \label{(11)}\\
f(0)\varphi_{1}(t)-a_{3}f(0)\varphi_{3}(t)-a_{3}\varphi(0)\beta_{3}(\eta)+\int_{0}^{t}K_{1}\alpha_{3}(\eta)d\eta=\nonumber\\
=\frac{1}{\Gamma(\frac{1}{3})}D_{(0,t)}^{\frac{2}{3}}[a_{3}F_{3}(0,t)-F_{1}(0,t)], \label{(12)}
\end{align}
where
$$
K_{1}=\int_{\eta}^{t}\frac{1}{(t-\tau)^{\frac{2}{3}}(\tau-\eta)^{\frac{1}{3}}}f'(-\frac{L}{(\tau-\eta)^{\frac{1}{3}}})d\eta.
$$
Analogously, letting
\begin{equation*}
K_{2}=\int_{\eta}^{t}\frac{1}{(t-\tau)^{\frac{1}{3}}(\tau-\eta)^{\frac{2}{3}}}f'(-\frac{L}{(\tau-\eta)^{\frac{1}{3}}})d\eta,
\end{equation*}
from the second part of  condition \eqref{(2)} we get
\begin{align}
\nonumber f'(0)\varphi_{1}(t)-b_{2}f'(0)\varphi_{2}(t)-b_{2}\varphi'(0)\beta_{2}(\eta)+\int_{0}^{t}K_{2}\alpha_{2}(\eta)d\eta=\\
=\frac{1}{\Gamma(\frac{2}{3})}D_{(0,t)}^{\frac{1}{3}}[a_{2}F_{2x}(0,t)-F_{1x}(0,t)], \label{(13)}
\end{align}
and
\begin{align}
\nonumber f'(0)\varphi_{1}(t)-b_{3}f'(0)\varphi_{3}(t)-b_{3}\varphi'(0)\beta_{3}(\eta)+\int_{0}^{t}K_{2}\alpha_{3}(\eta)d\eta=\\
=\frac{1}{\Gamma(\frac{2}{3})}D_{(0,t)}^{\frac{1}{3}}[a_{3}F_{3x}(0,t)-F_{1x}(0,t)], \label {(14)}
\end{align}
Also, defining
\[
K_{3}=\int_{0}^{x}U(y+L;t-\eta)dy,
\]
using the vertex condition \eqref{(3)}
%$u_{1xx}(0,t)=\frac{1}{a_{2}}u_{2xx}(0,t)+\frac{1}{a_{3}}u_{3xx}(0,t)$
we obtain
\begin{align}
\nonumber -\frac{\pi}{3}\varphi_{1}(t)-\frac{1}{a_{2}}\frac{2\pi}{3}\varphi_{2}(t)-\frac{1}{a_{3}}\frac{2\pi}{3}\varphi_{3}(t)+\int_{0}^{t}K_{3}\alpha_{2}(\eta)d\eta=\\
\frac{1}{a_{3}}F_{3xx}(0,t)+\frac{1}{a_{2}}F_{2xx}(0,t)-F_{1xx}(0,t), \label{(15)}
\end{align}
Satisfying the conditions \eqref{(4)}, we have:
\begin{align*}
\int_{0}^{t}\frac{a_{2}}{(t-\eta)^{\frac{1}{3}}}f(\frac{L}{(t-\eta)^\frac{1}{3}})\varphi_{2}(\eta)d\eta+&\int_{0}^{t}\frac{a_{2}}{(t-\eta)^{\frac{1}{3}}}f(0)\alpha_{2}(\eta)d\eta+\\
&\int_{0}^{t}\frac{a_{2}}{(t-\eta)^{\frac{1}{3}}}\varphi(\frac{L}{(t-\eta)^\frac{1}{3}})\beta_{2}(\eta)d\eta+a_{2}F_{2}(L,t)=\\
\int_{0}^{t}\frac{a_{4}}{(t-\eta)^{\frac{1}{3}}}f(0)\varphi_{4}(\eta)d\eta &+\int_{0}^{t}\frac{a_{4}}{(t-\eta)^{\frac{1}{3}}}\varphi(0)\psi_{4}(\eta)d\eta+a_{4}F_{4}(0,t)\\
\int_{0}^{t}\frac{a_{2}}{(t-\eta)^{\frac{1}{3}}}&f(\frac{L}{(t-\eta)^\frac{1}{3}})\varphi_{2}(\eta)d\eta+\int_{0}^{t}\frac{a_{2}}{(t-\eta)^{\frac{1}{3}}}f(0)\alpha_{2}(\eta)d\eta+\\
\int_{0}^{t}\frac{a_{2}}{(t-\eta)^{\frac{1}{3}}}\varphi(\frac{L}{(t-\eta)^\frac{1}{3}})\beta_{2}(\eta)d\eta+&a_{2}F_{2}(L,t)=\\
\int_{0}^{t}\frac{a_{5}}{(t-\eta)^{\frac{1}{3}}}f(0)&\varphi_{5}(\eta)d\eta+\int_{0}^{t}\frac{a_{5}}{(t-\eta)^{\frac{1}{3}}}\varphi(0)\psi_{5}(\eta)d\eta+a_{5}F_{5}(0,t)
\end{align*}
 Taking derivatives of order  1/3 of the above equation and letting
$$
K_{4}=\int_{\eta}^{t}\frac{1}{(t-\tau)^{\frac{2}{3}}(\tau-\eta)^{\frac{1}{3}}}f'(\frac{L}{(\tau-\eta)^{\frac{1}{3}}})d\eta,
$$
we obtain
\begin{align}
\nonumber a_{2}f(0)\alpha_{2}(t)-a_{4}f(0)\varphi_{4}(t)-a_{4}\varphi(0)\psi_{4}(t)+\int_{0}^{t}K_{4}\varphi_{2}(\eta)d\eta+
+\int_{0}^{t}K_{4}\beta_{2}(\eta)d\eta= \\
\label{(16)}=\frac{1}{\Gamma(\frac{1}{3})}D_{(0,t)}^{\frac{2}{3}}[a_{4}F_{4}(0,t)-a_{2}F_{2}(L,t)],\\
\nonumber  a_{2}f(0)\alpha_{2}(t)-a_{5}f(0)\varphi_{5}(t)-a_{5}\varphi(0)\psi_{5}(t)+\int_{0}^{t}K_{4}\varphi_{2}(\eta)d\eta+\int_{0}^{t}K_{4}\beta_{2}(\eta)d\eta=\\
\label{(17)}=\frac{1}{\Gamma(\frac{1}{3})}D_{(0,t)}^{\frac{2}{3}}[a_{5}F_{5}(0,t)-a_{2}F_{2}(L,t)].
\end{align}
Analogously, from the second part of \eqref{(4)} we get
\begin{align}
\nonumber b_{2}f'(0)\alpha_{2}(t)-b_{4}f'(0)\varphi_{4}(t)-b_{4}\varphi'(0)\psi_{4}(t)+\int_{0}^{t}K_{5}\varphi_{2}(\eta)d\eta+\int_{0}^{t}K_{5}\beta_{2}(\eta)d\eta=\\
\label{(18)}=\frac{1}{\Gamma(\frac{2}{3})}D_{(0,t)}^{\frac{1}{3}}[b_{4}F_{4x}(0,t)-b_{2}F_{2x}(L,t)],\\
\nonumber b_{2}f'(0)\alpha_{2}(t)-b_{5}f'(0)\varphi_{5}(t)-b_{5}\varphi'(0)\psi_{5}(t)+\int_{0}^{t}K_{5}\varphi_{2}(\eta)d\eta
+\int_{0}^{t}K_{5}\beta_{2}(\eta)d\eta=\\
\label{(19)}=\frac{1}{\Gamma(\frac{2}{3})}D_{(0,t)}^{\frac{1}{3}}[b_{5}F_{5x}(0,t)-b_{2}F_{2x}(L,t)].
\end{align}
Using the vertex condition \eqref{(5)}
%$\frac{1}{a_{2}}u_{2xx}(L,t)=\frac{1}{a_{4}}u_{4xx}(0,t)+\frac{1}{a_{5}}u_{5xx}(0,t)$
we obtain
\begin{equation}\label{(20)}
\begin{aligned}
 -\frac{1}{a_{2}}\frac{\pi}{3}\alpha_{2}(t)-\frac{1}{a_{4}}\frac{2\pi}{3}\varphi_{4}(t)-\frac{1}{a_{5}}\frac{2\pi}{3}\varphi_{5}(t)+\int_{0}^{t}K_{6}\varphi_{2}(\eta)d\eta=\\
=\frac{1}{a_{4}}F_{4xx}(0,t)+\frac{1}{a_{5}}F_{5xx}(0,t)-\frac{1}{a_{2}}F_{2xx}(L,t).
\end{aligned}
\end{equation}
Satisfying the conditions \eqref{(6)}, we have:
$$\int_{0}^{t}\frac{a_{3}}{(t-\eta)^{\frac{1}{3}}}f(\frac{L}{(t-\eta)^\frac{1}{3}})\varphi_{3}(\eta)d\eta+\int_{0}^{t}\frac{a_{3}}{(t-\eta)^{\frac{1}{3}}}f(0)\alpha_{3}(\eta)d\eta+$$
$$\int_{0}^{t}\frac{a_{3}}{(t-\eta)^{\frac{1}{3}}}\varphi(\frac{L}{(t-\eta)^\frac{1}{3}})\beta_{3}(\eta)d\eta+a_{3}F_{3}(L,t)=$$
$$\int_{0}^{t}\frac{a_{6}}{(t-\eta)^{\frac{1}{3}}}f(0)\varphi_{6}(\eta)d\eta+\int_{0}^{t}\frac{a_{6}}{(t-\eta)^{\frac{1}{3}}}\varphi(0)\psi_{6}(\eta)d\eta+a_{6}F_{6}(0,t)$$
$$\int_{0}^{t}\frac{a_{3}}{(t-\eta)^{\frac{1}{3}}}f(\frac{L}{(t-\eta)^\frac{1}{3}})\varphi_{3}(\eta)d\eta+\int_{0}^{t}\frac{a_{3}}{(t-\eta)^{\frac{1}{3}}}f(0)\alpha_{3}(\eta)d\eta+$$
$$\int_{0}^{t}\frac{a_{3}}{(t-\eta)^{\frac{1}{3}}}\varphi(\frac{L}{(t-\eta)^\frac{1}{3}})\beta_{3}(\eta)d\eta+a_{3}F_{3}(L,t)=$$
$$\int_{0}^{t}\frac{a_{7}}{(t-\eta)^{\frac{1}{3}}}f(0)\varphi_{7}(\eta)d\eta+\int_{0}^{t}\frac{a_{7}}{(t-\eta)^{\frac{1}{3}}}\varphi(0)\psi_{7}(\eta)d\eta+a_{7}F_{7}(0,t).$$
Using properties of fractional derivatives, one can obtain the following equalities
\begin{align}
\nonumber a_{3}f(0)\alpha_{3}(t)-a_{6}f(0)\varphi_{6}(t)-a_{6}\varphi(0)\psi_{6}(t)+\int_{0}^{t}K_{4}\varphi_{3}(\eta)d\eta
 +\int_{0}^{t}K_{4}\beta_{3}(\eta)d\eta\\
  \label{(21)}=\frac{1}{\Gamma(\frac{1}{3})}D_{(0,t)}^{\frac{2}{3}}[a_{6}F_{6}(0,t)-a_{3}F_{3}(L,t)],\\
\nonumber a_{3}f(0)\alpha_{3}(t)-a_{7}f(0)\varphi_{7}(t)-a_{7}\varphi(0)\psi_{7}(t)+\int_{0}^{t}K_{4}\varphi_{3}(\eta)d\eta+\int_{0}^{t}K_{4}\beta_{3}(\eta)d\eta=\\
\label{(22)} \frac{1}{\Gamma(\frac{1}{3})}D_{(0,t)}^{\frac{2}{3}}[a_{7}F_{7}(0,t)-a_{3}F_{3}(L,t)]
\end{align}
Analogously,  letting
\[
K_{5}=\int_{\eta}^{t}\frac{1}{(t-\tau)^{\frac{1}{3}}(\tau-\eta)^{\frac{2}{3}}}f''(\frac{L}{(\tau-\eta)^{\frac{1}{3}}})d\eta,
\]
from the second part of this condition we get
\begin{align}
\nonumber b_{3}f'(0)\alpha_{3}(t)-b_{6}f'(0)\varphi_{6}(t)-b_{6}\varphi'(0)\psi_{6}(t)+\int_{0}^{t}K_{5}\varphi_{3}(\eta)d\eta
+\int_{0}^{t}K_{5}\beta_{3}(\eta)d\eta\\
\label{(23)}=\frac{1}{\Gamma(\frac{2}{3})}D_{(0,t)}^{\frac{1}{3}}[b_{6}F_{6x}(0,t)-b_{3}F_{3x}(L,t)],\\
\nonumber b_{3}f'(0)\alpha_{3}(t)-b_{7}f'(0)\varphi_{7}(t)-b_{7}\varphi'(0)\psi_{7}(t)+\int_{0}^{t}K_{5}\varphi_{3}(\eta)d\eta
+\int_{0}^{t}K_{5}\beta_{3}(\eta)d\eta=\\
\label{(24)} \frac{1}{\Gamma(\frac{2}{3})}D_{(0,t)}^{\frac{1}{3}}[b_{7}F_{7x}(0,t)-b_{3}F_{3x}(L,t)].
\end{align}
 Finally, letting
$$ K_{6}=\int_{0}^{x}U(y-L;t-\eta)dy$$
from the vertex condition \eqref{(7)}
% $\frac{1}{a_{3}}u_{3xx}(L,t)=\frac{1}{a_{6}}u_{6xx}(0,t)+\frac{1}{a_{7}}u_{7xx}(0,t)$
we obtain
\begin{equation}\label{(25)}
\begin{aligned}
-\frac{1}{a_{3}}\frac{\pi}{3}\alpha_{3}(t)-\frac{1}{a_{6}}\frac{2\pi}{3}\varphi_{6}(t)-\frac{1}{a_{7}}\frac{2\pi}{3}\varphi_{7}(t)+\int_{0}^{t}K_{6}\varphi_{3}(\eta)d\eta=\\
=\frac{1}{a_{6}}F_{6xx}(0,t)+\frac{1}{a_{7}}F_{7xx}(0,t)-\frac{1}{a_{3}}F_{3xx}(L,t),
\end{aligned}
\end{equation}

We obtained the system of integral equations \eqref{(11)}-\eqref{(25)} with respect to unknowns
$$\Phi(t)=(\varphi_{m}(t),\psi_{n}(t),\alpha_{k}(t),\beta_{l}(t))^{T}, m=\overline{1,7},n=\overline{4,7},k=\overline{2,3},l=2,3$$
I-Bond matrix
\[{{A}_{1}}=\left( \begin{matrix}
   \begin{matrix}
   \begin{matrix}
   \begin{matrix}
   f(0)  \\
   f(0)  \\
\end{matrix}  \\
   {f}'(0)  \\
   {f}'(0)  \\
   -\frac{\pi }{3}  \\
\end{matrix} & \begin{matrix}
   \begin{matrix}
   -{{a}_{2}}f(0)  \\
   0  \\
\end{matrix}  \\
   -{{b}_{2}}{f}'(0)  \\
   0  \\
   -\frac{2\pi }{3{{a}_{2}}}  \\
\end{matrix}  \\
\end{matrix} & \begin{matrix}
   \begin{matrix}
   0  \\
   -{{a}_{3}}f(0)  \\
\end{matrix}  \\
   0  \\
   -{{b}_{3}}{f}'(0)  \\
   -\frac{2\pi }{3{{a}_{3}}}  \\
\end{matrix} & \begin{matrix}
   \begin{matrix}
   -{{a}_{2}}\varphi (0)  \\
   0  \\
\end{matrix}  \\
   -{{b}_{2}}{\varphi }'(0)  \\
   0  \\
   0  \\
\end{matrix} & \begin{matrix}
   \begin{matrix}
   0  \\
   -{{a}_{3}}\varphi (0)  \\
\end{matrix}  \\
   0  \\
   -{{b}_{3}}{\varphi }'(0)  \\
   0  \\
\end{matrix}  \\
\end{matrix} \right)\]
of the coefficients of these unknowns on the off integral part of the system
has a determinant
$$detA_{1}=-\frac{\pi^{3}}{27}(a_{2}b_{2}a_{3}b_{3}+\frac{a_{3}b_{3}(a_{2}+b_{2})}{a_{2}}+\frac{a_{2}b_{2}(a_{3}+b_{3})}{a_{3}})$$
II-Bond matrix
\[{{A}_{2}}=\left( \begin{matrix}
   \begin{matrix}
   \begin{matrix}
   \begin{matrix}
   {{a}_{2}}f(0)  \\
   {{a}_{2}}f(0)  \\
\end{matrix}  \\
   {{b}_{2}}{f}'(0)  \\
   {{b}_{2}}{f}'(0)  \\
   -\frac{\pi }{3}  \\
\end{matrix} & \begin{matrix}
   \begin{matrix}
   -{{a}_{4}}f(0)  \\
   0  \\
\end{matrix}  \\
   -{{b}_{4}}{f}'(0)  \\
   0  \\
   -\frac{2\pi }{3{{a}_{4}}}  \\
\end{matrix}  \\
\end{matrix} & \begin{matrix}
   \begin{matrix}
   0  \\
   -{{a}_{5}}f(0)  \\
\end{matrix}  \\
   0  \\
   -{{b}_{5}}{f}'(0)  \\
   -\frac{2\pi }{3{{a}_{5}}}  \\
\end{matrix} & \begin{matrix}
   \begin{matrix}
   -{{a}_{4}}\varphi (0)  \\
   0  \\
\end{matrix}  \\
   -{{b}_{4}}{\varphi }'(0)  \\
   0  \\
   0  \\
\end{matrix} & \begin{matrix}
   \begin{matrix}
   0  \\
   -{{a}_{5}}\varphi (0)  \\
\end{matrix}  \\
   0  \\
   -{{b}_{5}}{\varphi }'(0)  \\
   0  \\
\end{matrix}  \\
\end{matrix} \right)\]
of the coefficients of these unknowns on the off integral part of the system
has a determinant
$$detA_{2}=-\frac{\pi^{3}}{27}[a_{4}b_{4}a_{5}b_{5}+\frac{a_{5}b_{5}(a_{4}b_{2}+a_{2}b_{4})}{a_{4}}+\frac{a_{4}b_{4}(a_{5}b_{2}+a_{2}b_{5})}{a_{5}}]$$
III-Bond matrix
\[{{A}_{3}}=\left( \begin{matrix}
   \begin{matrix}
   \begin{matrix}
   \begin{matrix}
   {{a}_{3}}f(0)  \\
   {{a}_{3}}f(0)  \\
\end{matrix}  \\
   {{b}_{3}}{f}'(0)  \\
   {{b}_{3}}{f}'(0)  \\
   -\frac{\pi }{3}  \\
\end{matrix} & \begin{matrix}
   \begin{matrix}
   -{{a}_{6}}f(0)  \\
   0  \\
\end{matrix}  \\
   -{{b}_{6}}{f}'(0)  \\
   0  \\
   -\frac{2\pi }{3{{a}_{6}}}  \\
\end{matrix}  \\
\end{matrix} & \begin{matrix}
   \begin{matrix}
   0  \\
   -{{a}_{7}}f(0)  \\
\end{matrix}  \\
   0  \\
   -{{b}_{7}}{f}'(0)  \\
   -\frac{2\pi }{3{{a}_{7}}}  \\
\end{matrix} & \begin{matrix}
   \begin{matrix}
   -{{a}_{6}}\varphi (0)  \\
   0  \\
\end{matrix}  \\
   -{{b}_{6}}{\varphi }'(0)  \\
   0  \\
   0  \\
\end{matrix} & \begin{matrix}
   \begin{matrix}
   0  \\
   -{{a}_{7}}\varphi (0)  \\
\end{matrix}  \\
   0  \\
   -{{b}_{7}}{\varphi }'(0)  \\
   0  \\
\end{matrix}  \\
\end{matrix} \right)\]
of the coefficients of these unknowns on the off integral part of the system
has a determinant
$$detA_{3}=-\frac{\pi^{3}}{27}[a_{6}b_{6}a_{7}b_{7}+\frac{a_{7}b_{7}(a_{6}b_{3}+a_{3}b_{6})}{a_{6}}+\frac{a_{6}b_{6}(a_{7}b_{3}+a_{3}b_{7})}{a_{7}}]$$
Thus under conditions of the theorem the matrix

\[A=\left( \begin{matrix}
   {{A}_{1}} & 0 & 0  \\
   0 & {{A}_{2}} & 0  \\
   0 & 0 & {{A}_{3}}  \\
\end{matrix} \right)\]
is non-singular. According to the asymptotics of Airy functions, the kernels of the integral operators are integrable. Hence, it follows from the uniqueness theorem and Fredholm alternatives that the system of equations has a unique solution. Thus the solvability of the problem is proved.
\end{proof}

\section{Applications of the model}
The model we have studied in this paper can be applied to the wave propagation in the system of water channels and oil pipelines with branching points. Here, we'll show the consistency of our model with a real physical system. 

Notice that usually, the length of pipelines are long enough so that,  the system can be viewed as a metric graph. In this case, functions $u_j$ in the KdV equations have the physical meaning of vertical displacement of the water particles from the equilibrium (unperturbed) state in the channel. To have a realistic vertex (i.e. branching point) boundary conditions, first of all, we have to have the continuity of the displacement $u_j$ at the branching points. Also, it is intuitively clear, that the slope (gradient) of the wave should be continuous at the branching point. Finally taking into account that we are interested in solutions that are rapidly decaying at infinity, we obtain that These imply, that all the parameters $a_j$ and $b_j$ of the boundary conditions (weights) in equation \eqref{(2)} should satisfy $1>a_2>a_3>0$, and $1>b_2>b_3>0$.  We refer to \cite[Remark 1, and Fig. 1]{SUA} for detailed description of these properties. 

Other (remaining) vertex boundary conditions have to be derived using fundamental conservation laws, like mass, energy and momentum conservations. We assume, that the soliton do not lose it's energy during the propagation, that is the common property of solitons. Energy for the wave, dynamics of which is described by KdV equation can be written as \cite{AK}
\begin{equation}
E=\sum_{B_k}\int_{B_k}u_k^2dx.
\end{equation}
By using the total energy conservation law, $dE/dt=0$, we can derive the vertex boundary conditions in equation \eqref{(3)}.

The above shows that  the model considered in this paper is realizable from the point of view of physics.

\end{document}